\documentclass[a4paper,11pt]{article}
\usepackage{amssymb,amsmath,amsthm}
\usepackage{psfig}
\usepackage{graphicx}
\usepackage{epstopdf}
\usepackage{caption}

\numberwithin{equation}{section}

\newcommand{\m}[1]{\mathbb{ #1}}
\newcommand{\mc}[1]{\mathcal{ #1}}

\newfont{\goth}{eufm10 at 12pt}
\newfont{\gots}{eufm8 at 9pt}
\newcommand{\rmd}{{\,\rm d}}

\def\al{\alpha}               \def\ga{\gamma}
\def\de{\delta}       \def\eps{\varepsilon}

\def\si{\sigma}                
\def\ph{\varphi}               
       \def\Ga{\Gamma}

\newtheorem{Thm}{Theorem}[section]
\newtheorem{Prop}[Thm]{Proposition}
\newtheorem{Lem}[Thm]{Lemma}
\newtheorem{Cor}[Thm]{Corollary}

\newtheorem{Rem}[Thm]{Remark}
\newtheorem{Fact}[Thm]{Fact}
\newtheorem{Def}[Thm]{Definition}
\def\bt{\begin{Thm}}
\def\et{\end{Thm}}
\def\br{\begin{Rmq}}
\def\er{\end{Rmq}}
\def\bc{\begin{Cor}}
\def\ec{\end{Cor}}
\def\bp{\begin{Prop}}
\def\ep{\end{Prop}}
\def\bl{\begin{Lem}}
\def\el{\end{Lem}}
\def\bd{\begin{Def}}
\def\ed{\end{Def}}
\def\bq{\begin{quotation}}
\def\eq{\end{quotation}}
\def\bfa{\begin{Fact}}
\def\efa{\end{Fact}}
\def\ra{\rightarrow}
\def\vs{\vspace{1em}}

\def\noi{\noindent}

\title{Quasiisometric harmonic maps\\ between rank one symmetric spaces}
\author{Yves Benoist \& Dominique Hulin}
\date{}

\begin{document}
\maketitle
\begin{abstract}
We prove that a quasiisometric map between rank one symmetric spaces 
is within bounded distance from a unique harmonic map. In particular, this completes the proof of the Schoen-Li-Wang conjecture.
\end{abstract}


\renewcommand{\thefootnote}{\fnsymbol{footnote}} 
\renewcommand{\thefootnote}{\arabic{footnote}} 

\section{Introduction}  
\label{secintro}

\subsection{Main result}
\label{secmain}

We first  explain  the title.

A {\it symmetric space}  is a Riemannian manifold $X$ such that, for every point $x$ in $X$,
there exists a symmetry centered at $x$, i.e. an isometry $s_x$ of $X$ fixing the point $x$ and whose differential at $x$ is minus the identity.

In this article we will call {\it rank one symmetric space} 
a symmetric space $X$ whose sectional curvature is everywhere negative~: $K_X<0$.
The list of rank one symmetric spaces is well known. They are the real hyperbolic spaces $\m H^p_{\m R}$, the complex hyperbolic spaces
$\m H^p_\m C$, the quaternion hyperbolic spaces $\m H^p_\m Q$, with $p\geq 2$, 
and the Cayley hyperbolic plane
$\m H^2_{\m C a}$. 

A map $f:X\ra Y$ between two metric spaces $X$ and $Y$ 
is said to be {\it quasiisometric} if there exists
a constant $c\geq 1$ such that $f$ is {\it $c$-quasiisometric} i.e. such that, for all $x$, $x'$ in $X$,
one has
\begin{equation}
\label{eqnquasiiso}
c^{-1}\, d(x,x')-c
\;\leq\; 
d(f(x),f(x'))
\;\leq\; 
c\, d(x,x')+c\; .
\end{equation}

A map $h:X\ra Y$ between two Riemannian manifolds $X$ and $Y$ is said to be 
{\it harmonic} if its tension field is zero i.e. if it satisfies the elliptic nonlinear partial differential equation
${\rm tr}(D^2h)=0$ where $D^2h$ is the second covariant derivative of $h$. 
For instance, an isometric map with totally geodesic image is always harmonic.
The problem of the existence, regularity  and uniqueness of harmonic maps under various boundary conditions
is a very classical topic (see \cite{EeLem78}, \cite{Jost84},
\cite{donn94}, \cite{Simon96}, \cite{SchoenYau97} or \cite{LinWang08}).  
In particular, when $Y$ is simply connected and has non positive curvature, a harmonic map is always $\mc C^\infty$ i.e. it is indefinitely differentiable,
and is a minimum of the energy functional -see Formula \eqref{eqnenergy}- among maps that agree with $h$ outside a compact subset
of $X$.

The aim of this article is to prove the following

\bt
\label{thdfxhx}
Let $f:X\ra Y$ be a quasiisometric map between rank one symmetric spaces
$X$ and $Y$. Then there exists a unique harmonic map $h:X\ra Y$ which stays within bounded distance from $f$
i.e. such that
$$
\sup_{x\in X}d(h(x),f(x))<\infty\; .
$$
\et

The uniqueness of $h$ is due to Li and Wang in 
\cite[Theorem 2.3]{LiWang98}. In this article we prove the existence of $h$.

\subsection{Previous results and conjectures}
\label{secconj}

When $X$ is equal to $Y$, 
Theorem \ref{thdfxhx} was conjectured by Li and Wang in 
\cite[Introduction]{LiWang98} extending a conjecture of Schoen in \cite{Schoen90}
for the case $X=Y=\m H^2_\m R$. 
When these conjectures were formulated, the case where $X=Y$  is either equal 
to $\m H^p_\m Q$ or $\m H^2_{\m C a}$ was already known by a previous result of Pansu 
in \cite{Pansu89}. In that case the harmonic map $h$ is an onto isometry.

The uniqueness part of the Schoen conjecture was quickly settled by Li and Tan in \cite{LiTam93},
and the proof was extended by Li and Wang to rank one symmetric spaces in their paper \cite{LiWang98}.  
Partial results towards the existence statement in the Schoen conjecture
were obtained in \cite{TamWan95}, \cite{HardtWolf97}, \cite{Rivet99},
\cite{Markovic02}, \cite{BonSch10}.
A major breakthrough was then achieved by Markovic who proved successively 
the Li-Wang conjecture for the case $X=Y=\m H^3_\m R$ in \cite{Marko3}, 
for the case $X=Y=\m H^2_\m R$ in \cite{Marko2} thus solving the initial Schoen conjecture,
and very recently with Lemm for the case $X=Y=\m H^p_\m R$ with $p\geq 3$ in \cite{Markon}.

As a corollary of Theorem \ref{thdfxhx}, we complete the proof of the Li-Wan conjecture.
In particular, we obtain the following. 

\bc
\label{cordfxhx1}
For $p\geq 1$, any quasiisometric map $f:\m H^p_\m C\ra\m H^p_\m C$
is within bounded distance from a unique harmonic map $h:\m H^p_\m C\ra\m H^p_\m C$.
\ec

Another new feature in Theorem \ref{thdfxhx} is that one does not assume 
that $X$ and $Y$ have the same dimension. Even the following special case is new.

\bc
\label{cordfxhx2}
Any quasiisometric map $f:\m H^2_\m R\ra\m H^3_\m R$
is within bounded distance from a unique harmonic map $h:\m H^2_\m R\ra\m H^3_\m R$.
\ec

We finally recall that,
according to a well-known result of Kleiner and Leeb in \cite{KleinerLeeb97}, every
quasiisometric map $f:X\ra Y$ between irreducible higher rank symmetric spaces stays within bounded distance of an isometric map, after a suitable scalar rescaling of the metrics.
Another proof of this result has also been given by Eskin and Farb in \cite{EskinFarb97}.

\subsection{Motivation}
\label{secmotiv}

We now briefly recall  a few definitions and facts that are useful 
to understand the context and the motivation of Theorem \ref{thdfxhx}.
None of them will be used in the other sections of this article.

Let $X$ and $Y$ be rank one symmetric spaces.
Recall first that  $X$ is diffeomorphic to 
$\m R^k$ and has a visual compactification $X\cup\partial X$. 
The visual boundary $\partial X$
is homeomorphic to 
a topological sphere $\m S^{k-1}$. 
Choosing a base point $O$ in $X$, this boundary is endowed with the Gromov quasidistance $d'$
defined by $d'(\xi,\eta):=e^{-(\xi|\eta)_{O}} $
for $\xi$, $\eta$ in $\partial X$ where 
$(\xi|\eta)_{O}$ denotes the Gromov product
(see \cite[Sec. 7.3]{GhysHarp90}). 
A non-constant continuous map $F:\partial X\ra\partial Y$ between the boundaries
is called {\it quasisymmetric} if there exists $K\geq 1$
such that for all $\xi$, $\eta$, $\zeta$ in $\partial X$ with $d'(\xi,\eta)\leq d'(\xi,\zeta)$,
one has $d'(F(\xi),F(\eta)\leq K d'( F(\xi),F(\zeta))$. 

The following nice fact, 
which is also true for a wider class of geodesic Gromov hyperbolic spaces, 
gives another point of view on quasiisometric maps~:

\bfa
\label{factxdx}
Let $X$, $Y$ be rank one symmetric spaces.\\
$a)$ Any quasiisometric map $f:X\ra Y$ induces a boundary map
$\partial f:\partial X\ra\partial Y$ which is quasisymmetric.\\
$b)$ Two quasiisometric maps $f, g:X\ra Y $ have the same boundary map 
$\partial f=\partial g$ if and only if  $f$ and $g$ 
are within bounded distance from one another.\\
$c)$ Any quasisymmetric map $F :\partial X\ra\partial Y$ is the boundary map 
$F=\partial f$ of a quasiisometric map $f:X\rightarrow Y$.
\efa

This fact has a long history. Point $a$  is in the paper of Mostow \cite{Mostow73} 
and was extended later by Gromov in \cite[Sec. 7]{Gromov87} 
(see also  \cite[Sec. 7]{GhysHarp90}). Point $b$
is in the paper of Pansu \cite[Sec. 9]{Pansu89}. Point $c$ 
is in the paper of Bourdon and Pajot  \cite[Section 2.2]{BourdonPajot03}
extending previous results of Tukia in \cite{Tukia85},
of Paulin in \cite{Paulin96},
and of Bonk, Heinonen, Koskela in \cite{BonkHeiKos01}.

Recall that a diffeomorphism $f$ of $X$ is said to be {\it quasiconformal} 
if the function $x\mapsto \|Df(x)\|\|Df(x)^{-1}\|$ is uniformly bounded on $X$. 
The original formulation of the Schoen conjecture involved quasiconformal diffeomorphisms
instead of quasiisometries~:
{\it Every quasisymmetric homeomorphism F of $\m S^1$ is the boundary map
of a unique quasiconformal harmonic diffeomorphism of $\m H^2_\m R$}.

Relying on a previous result of Wan  in \cite[Theorem 13]{Wan92}, Li and Wang  pointed out in \cite[Theorem 1.8]{LiWang98} that 
{\it a harmonic map between the hyperbolic plane $\m H^2_\m R$
and itself is a quasiconformal diffeomorphism 
if and only if it is a quasiisometric map}.
This is why
Li and Wang  formulated in \cite{LiWang98} the higher dimensional generalisation 
of the Schoen conjecture using quasiisometries instead of quasiconformal diffeomorphisms.

Note that  in dimension  $k\geq 6$ 
there exist  harmonic maps 
between compact manifolds of negative curvature
which are homotopic to a diffeomorphism 
but which are not diffeomorphisms (see  \cite{FarrellOR00}).

\subsection{Strategy}
\label{secstrategy}

To prove our Theorem \ref{thdfxhx}, we start with 
a $c$-quasiisometric map $f:X\ra Y$ between rank one symmetric spaces.
We want to exhibit a harmonic map $h:X\ra Y$ within bounded distance from $f$. 

We will first gather in Chapter \ref{sechadamani} 
a few properties of Hada\-mard manifolds~:
images of triangles under quasiisometric maps, Hessian of the distance function,
gradient estimate for functions with bounded 
Laplacian.\vs

The first key point in our proof is the simple remark 
that, thanks to a smoothing process, we may assume without loss of generality
that the $c$-quasiisometric map $f$ is $\mc C^\infty$ and that its first and second covariant derivatives $Df$ and $D^2f$
are uniformly bounded on $X$ (Proposition \ref{prosmoothquasi}). 
We fix a point $O$ in $X$. For $R>0$, we denote by $B_{_R}:=B(O,R)$ 
the closed ball in $X$ with center $O$ and radius $R$
and by $\partial B_{_R}$ the sphere that bounds $B_{_R}$. 
We introduce the unique harmonic map $h_{_R}:B_{_R}\ra Y$ whose restriction to the sphere
$\partial B_{_R}$ is equal to $f$. This map $h_{_R}$ is $\mc C^\infty$ on the closed ball $B_{_R}$.
The harmonic  map $h$ will be constructed as the limit of  the maps $h_{R}$    
when $R$ goes to infinity.
In order to prove the existence of this limit $h$,
using  a classical compactness argument that we will recall in Section
\ref{secexistharm}, 
we just have to check that on the balls $B_{_R}$ the distances 
$$
\rho_{_R}:=d(h_{_R},f)
$$ 
are uniformly bounded in $R$. 
We will argue by contradiction and assume that 
we can find radii $R$ with $\rho_{_R}$ arbitrarily large.\vs

The second key point in our proof is what we call the boundary estimate 
(Proposition \ref{proboundary}).
It tells us that the ratio 
$\displaystyle
\frac{d(h_{_R}(x),f(x))}{d(x,\partial B_{_R})}
$
is uniformly bounded 
for $R\geq 1$ and $x$ in $B_{_R}$. 
In particular, when $\rho_{_R}$ is large, the ball
$B(O,R\!-\!1)$ contains a ball $B(x_{_R},r_{_R})$ whose center 
$x_{_R}$ satisfies $d(h_{_R}(x_{_R}),f(x_{_R}))=\rho_{_R}$ and 
whose radius $r_{_R}\geq 1$
is quite large. A good choice for the radius $r_{_R}$ will be
$r_{_R}=\rho_{_R}^{1/3}$.
We will focus on the restriction of the  maps $f$ and $h_{_R}$ to this ball $B(x_{_R},r_{_R})$. Let $y_{_R}:=f(x_{_R})$.
For $z$ in  $B(x_{_R},r_{_R})$, we will write 
$$
f(z)={\rm exp}_{y_{_R}}(\rho_f(z) v_f(z))
\;\;\;{\rm and}\;\;\;
h_{_R}(z)={\rm exp}_{y_{_R}}(\rho_h(z) v_h(z)),
$$
where $\rho_f(z)$, $\rho_h(z)$ are non negative and 
$v_f(z)$, $v_h(z)$ lie  in the unit sphere
$T^1_{y_{_R}}Y$ of the tangent space $T_{y_{_R}}Y$. 
We write $v_{_R}:=v_h(x_{_R})$ and we denote by 
$\theta(v_1, v_2)$ the angle between two vectors $v_1$, $v_2$ of the sphere $T^1_{y_{_R}}Y$.\vs

The third key point in our proof 
is to write for each point $z$ on the sphere $S(x_{_R},r_{_R})$ 
the triangle inequality 
\begin{equation*}
\theta(v_f(z),v_{_R})\leq \theta(v_f(z),v_h(z)) + \theta(v_h(z),v_{_R})
\end{equation*}
and, adapting an idea of Markovic in \cite{Marko2},  to focus on the set 
$$
W_{_R}:=\{z\in S(x_{_R},r_{_R})
\mid \rho_h(z)\geq \rho_{_R}\!-\! \frac{r_{_R}}{2c}
\;{\rm and}\;
\rho_h(z_t)\geq\frac{\rho_{_R}}{2} 
\;\mbox{\rm for}\; 
0\leq t\leq r_{_R} \}
$$
where $(z_t)_{0\leq t\leq r_{_R}}$ is the geodesic
segment between $x_{_R}$ and $z$.

\begin{figure}[ht]
\centerline{\includegraphics[width=12cm]{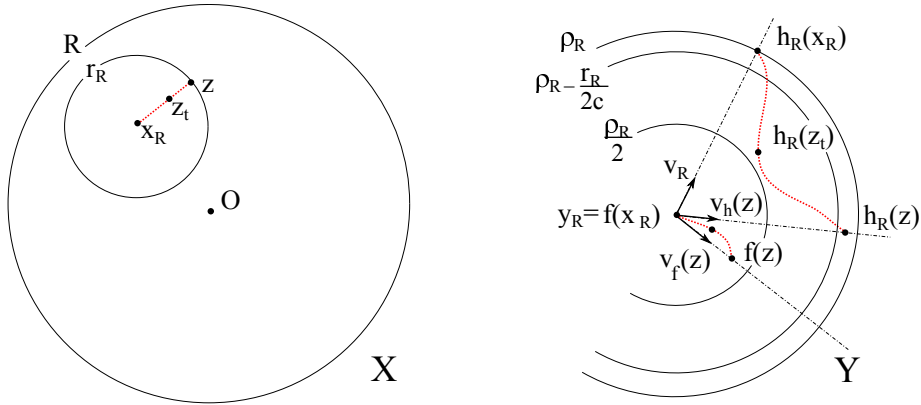}}
\caption{ {\small The vectors $v_f(z)$, $v_h(z)$ and $v_{_R}$ when $z$ belongs to $W_{_R}$.}} 
\label{fig:existharm}
\end{figure}

The contradiction will come from the fact that when both $R$ and $\rho_{_R}$
go to infinity,   the two angles 
$\theta_1:=\theta(v_f(z),v_h(z))$ and $\theta_2:=\theta(v_h(z),v_{_R})$ converge to $0$
uniformly for $z$ in $W_{_R}$,
while one can find $z=z_{_R}$ in $W_{_R}$ such that 
the other angle $\theta_0=\theta(v_f(z),v_{_R})$ stays away from $0$.
Here is a rough sketch of the arguments used to estimate these three angles.

To get the upper bound for the angle $\theta_1$ (Lemma \ref{lemI1}), 
we use the 
relation between angles and Gromov products (Lemma \ref{lemcomparison}) and 
we notice that the set $W_{_R}$ has been chosen so that
the Gromov product $(f(z)|h_{_R}(z))_{y_{_R}}$ is large.

To get the upper bound for the angle $\theta_2$ (Lemma \ref{lemI2}), 
we check that 
the gradient $Dv_h$ is uniformly small on the geodesic segment between $x_{_R}$ and $z$.
This follows from  the comparison inequality
$\;\; 2\,{\rm sinh}(\rho_h/2)\,\| Dv_h\|\leq \| Dh_{_R}\|\;$,
from the bound for $\| Dh_{_R}\|$ that is due to Cheng (Lemma \ref{lemcheng}), and from the definition of $W_{_R}$ 
that ensures that 
the factor ${\rm sinh}(\rho_h/2)$ stays very large
on this geodesic.

To find a point $z=z_{_R}$ in $W_{_R}$ such that the angle $\theta_0$ 
is not small (Lemma \ref{lemI0}), 
we use the almost invariance
of the Gromov products -and hence of the angles- under 
a quasiisometric map (Lemma \ref{lemproduct}).
We also use a uniform lower bound on the measure of $W_{_R}$ (Lemma \ref{lemsiwr}). 
This lower bound is a consequence of the subharmonicity of the function $\rho_h$  
(Lemma \ref{lemdefh}) and of
Cheng's estimate.
\vs

We would now like to point out the difference 
between our approach and those of the previous papers. 
The starting point of Markovic's method in \cite{Marko2}
is the fact that any $K$-quasi\-symmetric homeomorphism 
of the circle is a uniform limit of $K$-quasi\-symmetric diffeomorphisms. 
This fact  has no known analog in high dimension.
The starting point of the methods in both \cite{Marko3} and\cite{Markon}
is the fact that a quasisymmetric homeomorphism of 
the sphere $\m S^{k-1}$ is almost surely differentiable.
This fact is not true on $\m S^{1}$. 
Since our strategy avoids the use of quasisymmetric maps,
it gives a unified approach for all $\m H^p_\m R$ with $p\geq 2$,  
and also works  when $X$ and $Y$ have different dimensions.
\vs

Both authors thank the MSRI for its hospitality,
the Simons Foundation and the GEAR Network for their support, in Spring 2015 at the
beginning of this project.

\section{Hadamard manifolds}
\label{sechadamani}

In this preliminary chapter, we recall various estimates on a Hada\-mard manifold~: for the 
angles of a geodesic triangle,
for the  Hessian of the distance function,
and also for functions with bounded Laplacian.

\subsection{Triangles and quasiisometric maps}  
\label{secgeotriangle}
\bq
We first recall basic estimates for triangles in 
Hada\-mard manifolds
and explain how one controls the angles of the image of a triangle 
under a quasiisometric map.
\eq

All the Riemannian manifolds will be assumed to be connected 
and to have dimension at least two. We will denote by 
$d$ their distance function. 

A Hadamard manifold  is a complete simply connected Riemannian manifold $X$ of 
non positive curvature $K_X\leq 0$.  
For instance, the Euclidean space $\m R^k$ is a Hadamard manifold with zero curvature $K_X=0$,
while the rank one symmetric spaces are Hadamard manifolds with negative curvature $K_X<0$.
We will always assume without loss of generality 
that the metric on a rank one symmetric space  $X$
is normalized so that 
$-1\leq K_X\leq -1/4$.

Let  $x_0$, $x_1$, $x_2$ be three points on a Hadamard manifold $X$.
The {\it Gromov product} of the points $x_1$ and $x_2$ seen from $x_0$ is defined as
\begin{equation}
(x_1|x_2)_{x_0}:= (d(x_0,x_1)+d(x_0,x_2)-d(x_1,x_2))/2.
\end{equation}

We recall the basic comparison lemma which is one of the motivations for introducing the Gromov product.

\bl
\label{lemcomparison}
Let $X$ be a Hadamard manifold with
$-1\leq K_X\leq -a^2< 0$,
let $T$ be a geodesic triangle in $X$ with vertices $x_0$, $x_1$, $x_2$,
and let $\theta_0$ be the angle of $T$ at the vertex $x_0$.\\
$a)$ One has\;\; $(x_0|x_2)_{x_1}\geq d(x_0,x_1)\,\sin^2(\theta_0/2)$.\\
$b)$ One has \;\;
$\theta_0\leq 4\, e^{-a\,(x_1|x_2)_{x_0}} .$\\
$c)$ Moreover, when $\min( (x_0|x_1)_{x_2},(x_0|x_2)_{x_1})\geq 1$, one has 
$\theta_0\geq e^{-(x_1|x_2)_{x_0}}$.
\el

\begin{proof}
Assume first that $X$ is the hyperbolic plane $\m H^2_\m R$ with curvature $-1$. Let
$\ell_0:=d(x_1,x_2)$, $\ell_1:=d(x_0,x_1)$, $\ell_2:=d(x_0,x_2)$ be the side lengths of  $T$
and $m:=(\ell_1+\ell_2-\ell_0)/2$ so that
$$
(x_1|x_2)_{x_0}=m
\; ,\;\;
(x_0|x_2)_{x_1}=\ell_1-m
\; ,\;\;
(x_0|x_1)_{x_2}=\ell_2-m\; .
$$
The hyperbolic triangle equality reads as
\begin{equation}
\label{eqnsinshsh}
\sin^2(\theta_0/2) =
\frac{{\rm sinh}(\ell_1-m)}{{\rm sinh}(\ell_1)}
\,\frac{{\rm sinh}(\ell_2-m)}{{\rm sinh}(\ell_2)}\, .
\end{equation}
In particular, one has $\sin^2(\theta_0/2) \leq \frac{\ell_1-m}{\ell_1}$.
This proves Point $a$.

Points $b$ and $c$ follow from \eqref{eqnsinshsh} and the basic inequalities~:
\begin{eqnarray*}
\frac{{\rm sinh}(\ell-m)}{{\rm sinh}(\ell)}\leq e^{-m}
&\mbox{\rm  for}& 
0\leq m\leq \ell\, ,\\
\frac{{\rm sinh}(\ell-m)}{{\rm sinh}(\ell)}\geq  e^{-m} /2
&\mbox{\rm  for}&
0\leq m\leq \ell-1\, ,\\
\theta_0/4\leq \sin(\theta_0/2)\leq \theta_0/2 
&\mbox{\rm  for}& 
0\leq\theta_0\leq\pi.
\end{eqnarray*}

When the sectional curvature of $X$ is pinched between $-1$ and $-a^2$,
the triangle comparison theorems of Alexandrov and Toponogov (see \cite[Theorems 4.1 and 4.2]{Karcher89}) 
ensure that the result also holds in $X$.
\end{proof}

We now recall the effect of a quasiisometric map on the Gromov product.

\bl
\label{lemproduct}
Let $X$, $Y$ be Hadamard manifolds with
$-b^2\leq K_X\leq -a^2< 0$ and $-b^2\leq K_Y\leq -a^2< 0$,
and let $f:X\ra Y$ be a $c$-quasiisometric map.
There exists $A=A(a,b,c)>0$ such that, for all 
$x_0$, $x_1$, $x_2$ in $X$, one has
\begin{equation}
\label{eqnproduct}
c^{-1}(x_1|x_2)_{x_0}-A
\;\leq \;(
f(x_1)|f(x_2))_{f(x_0)}
\;\leq \;
c(x_1|x_2)_{x_0}+A\; .
\end{equation}
\el

\begin{proof}
This is a general property of quasiisometric maps between Gromov $\de$-hyperbolic spaces which is due to M. Burger. See \cite[Prop. 5.15]{GhysHarp90}.
\end{proof}

\subsection{Hessian of the distance function}  
\label{sechessiandist}

\bq
We now recall basic estimates for the Hessian of the distance function and of its square on a Hadamard manifold.
\eq

When $x_0$ is a point in a Riemannian manifold $X$, we denote by $d_{x_0}$ the distance function 
defined by $d_{x_0}(x)=d(x_0,x)$ for $x$ in $X$.
We denote  by $d^2_{x_0}$ the square of this function.
When $F$ is a $\mc C^2$ function on $X$, 
we denote by $DF$ the differential of $F$ 
and by $D^2F$ the Hessian of $F$,
which is by definition the second covariant derivative of $F$.

\bl
\label{lemhessiandist}
Let $X$ be a Hadamard manifold and $x_0\in X$.\\
$a)$ The Hessian of the square $d_{x_0}^2$ satisfies on $X$ 
\begin{equation}
\label{eqnd2rhox2}
D^2d_{x_0}^2 \geq 2\, g_{_X}\, ,
\end{equation}
where   $g_{_X}$ is the Riemannian metric on $X$.\\
$b)$ Assume that  
$-b^2\leq K_X\leq -a^2<0$. The Hessian of the distance function $d_{x_0}$ satisfies on $X\setminus\{x_0\}$
\begin{equation}
\label{eqnd2rhox0}
a\coth (a\, d_{x_0})\, g_0\leq D^2 d_{x_0} \leq 
b \coth (b\, d_{x_0})\, g_0\, ,
\end{equation}
where $g_0:=g_X -Dd_{x_0}\otimes Dd_{x_0}$.
\el

\begin{proof} $a)$ When $X={\mathbb R}^k$ is the $k$-dimensional Euclidean space, one has $ D^2d_{x_0}^2=2\, g_{_X}$. The general statement follows from this model case and the Alexandrov triangle comparison theorem.\\ 
$b)$  Assume first that $X={\mathbb H}^2_{\m R}$ is the real hyperbolic plane with  curvature $-1$.
Using the expression $\cosh (\ell_t)=\cosh(\ell_0 ) \cosh (t)$ for the length $\ell _t$ of the hypothenuse of a right triangle with side lengths $\ell_0$ and $t$, one infers that
$$
D^2 d_{x_0} =  
\coth ( d_{x_0})\, g_0\, .
$$
The general statement follows by the same argument combined again with the Alexandrov and Toponogov triangle comparison theorems. 
\end{proof}

\subsection{Functions with bounded Laplacian}  
\label{secboundedlap}

\bq
We give  a bound for functions defined on balls of a  Hadamard manifold, when their Laplacian is bounded
and their boundary value is equal to $0$.
\eq

The Laplace-Beltrami operator $\Delta$ on a Riemannian manifold $X$
is defined as the trace of the Hessian.
In local coordinates, the Laplacian of a function $F$ reads as
\begin{equation*}
\label{eqnlaplacian}
\Delta F={\rm tr}(D^2F)=\frac{1}{V}\sum_{i,j=1}^k
\frac{\partial}{\partial x_i}(V\, g_{_X}^{ij}
\frac{\partial}{\partial x_j}F)
\end{equation*}
where $V=\sqrt{\det (g_{_X}^{ij})}$ is the volume density.
The function $F$ is said to be harmonic if $\Delta F=0$,
subharmonic if $\Delta F\geq 0$ and superharmonic if $\Delta F\leq 0$. The study of  harmonic functions on Hadamard manifolds
has been initiated by Anderson and Schoen in \cite{AndSchoen85}.

\bp
\label{proboundedlap}
Let $X$ be a Hadamard manifold with
$K_X\leq -a^2<0$. Let $O$ be a point of $X$ 
and $B_{_R}=B(O,R)$ be the closed ball with center $O$ and radius $R>0$.
Let $G$ be a $\mc C^2$ function on $B_{_R}$ and let $M>0$.
Assume that 
\begin{eqnarray}
\label{eqnboundedlap1}
|\Delta G|&\leq &M
\;\;\;\mbox{\rm on $B_{_R}$}\\
\nonumber
G&=& 0
\;\;\;\mbox{\rm on $\partial B_{_R}$}.
\end{eqnarray}
Then, for all $x$ in $B_{_R}$,  one has the upper bound
\begin{equation}
\label{eqnboundedlap2}
|G(x)|\leq \tfrac{M}{a}\,d(x,\partial B_{_R})\; .
\end{equation}
\ep

\noi {\bf Remark. }
The assumption  in Proposition \ref{proboundedlap} that the curvature is negative is essential.
Indeed, the function $G:=R^2- d_{O}^2$ on the ball $B_{_R}$ of the Euclidean space $X=\m R^k$
satisfies \eqref{eqnboundedlap1} with $M=2k$
but does not satisfy \eqref{eqnboundedlap2} because the gradient of $G$ at a point 
$x$ on the sphere $\partial B_{_R}$ has norm $2\, R$.
\vs

The proof of Proposition  \ref{proboundedlap} relies on the following.

\bl
\label{lemboundedlap}
Let $X$ be a Hadamard manifold with
$K_X\leq -a^2$ and $x_0$ be a point of $X$. 
Then, the function $d_{x_0}$ is subharmonic. More precisely,
the distribution $\Delta\, d_{x_0}\! -\! a$ is  non-negative.
\el

\begin{proof}[Proof of Lemma \ref{lemboundedlap}]
Since $X$ is a Hadamard manifold, Lemma \ref{lemhessiandist}
ensures that the function $d_{x_0}$ is $\mc C^\infty$ on 
$X\setminus\{ x_0\}$ and satisfies $\Delta\, d_{x_0}(x)\geq a$ for $x\neq {x_0}$.
It remains to check that the distribution 
$\Delta\, d_{x_0}\! -\! a$ is non-negative on $X$.
The function $d_{x_0}$ is the uniform limit when $\eps$ converges to  $0$ 
of the $\mc C^\infty$ functions 
$d_{x_0,\eps}:=(\eps^2+d_{x_0}^2)^{1/2}$.
One computes their Laplacian on $X\setminus\{ x_0\}$~:
$$
\Delta d_{x_0,\eps}=\frac{d_{x_0}}{(\eps^2+d_{x_0}^2)^{1/2}}\, \Delta d_{x_0} +
\frac{\eps^2}{(\eps^2+d_{x_0}^2)^{3/2}}\; .
$$
Hence, one has on $X\setminus\{ x_0\}$~:
$$
\Delta d_{x_0,\eps}\geq \frac{a\, d_{x_0}}{(\eps^2+d_{x_0}^2)^{1/2}}\; .
$$
Since both sides are continuous functions on $X$, this inequality 
also holds on $X$.
Since a limit of non negative distributions is non negative,
one gets the inequality
$\Delta d_{x_0}\geq a$ on $X$ by letting $\eps$ go to $0$.
\end{proof}

\begin{proof}[Proof of Proposition \ref{proboundedlap}]
According to Lemma \ref{lemboundedlap}, both functions 
$$
G_{\pm}:=\tfrac{M}{a}(R-d_{O})\pm G
$$
are superharmonic on $B_{_R}$, i.e. one has $\Delta G_\pm\leq 0$.
Since they vanish on the boundary $\partial B_{_R}$,
the maximum principle ensures that these functions $G_\pm$ are non-negative on 
the ball $B_{_R}$.
\end{proof}

\section{Harmonic maps}
\label{secharmonicmap}

In this chapter we begin the proof of Theorem \ref{thdfxhx}. 
We first recall basic facts satisfied by harmonic maps.
We then explain why we can assume our $c$-quasiisometric map $f$
to be $\mc C^\infty$ with bounded covariant derivatives.
We also explain why an upper bound on $d(h_{_R},f)$ 
implies the existence of the harmonic map $h$.
Finally we provide this upper bound 
near the boundary $\partial B_{_R}$.

\subsection{Harmonic maps and the distance function}  
\label{secharmonicdist}

\bq
In this section, we recall two useful facts satisfied by a harmonic map $h$~:
the subharmonicity of the functions $d_{y_0}\circ h$, and 
Cheng's estimate for the differential $ Dh$.
\eq

\begin{Def}
\label{defharmonic}
Let $h:X\ra Y$ be a $\mc C^\infty$ map between two Riemannian manifolds.
The tension field of $h$ is the trace of the second covariant derivative $\tau(h):= {\rm tr} D^2h$.
The map $h$ is said to be {\it harmonic} if $\tau(h)=0$.  
\end{Def}

The tension $\tau(h)$ is a {\it $Y$-valued vector field on $X$}, i.e. 
it is a section of the pulled-back of the tangent bundle $TY\ra Y$ under the map 
$h:X\ra Y$.

\bl
\label{lemdefh}
Let $h:X\ra Y$ be a harmonic $\mc C^\infty$ map between Hadamard manifolds.
Let $y_0\in Y$ 
and let $\rho_h:X\ra \m R$ be the function 
$\rho_h:=d_{y_0}\circ h$. 
Then the continuous function $\rho_h$ is subharmonic on $X$. 
\el

\begin{proof}
The proof is similar to the proof of Lemma \ref{lemboundedlap}. 
We first recall the formula for the Laplacian of a composed function.
Let $f :X\ra Y$ be a
$\mc C^\infty$ map and $F\in \mc C^\infty(Y)$ be a $\mc C^\infty$ function on $Y$. Then one has
\begin{equation}
\label{eqndefh}
\Delta (F\circ f) =\sum_{1\leq i\leq k} D^2F(D_{e_i}f,D_{e_i}f)+ \langle DF ,\tau(f)\rangle\, ,
\end{equation}
where 
$(e_i)_{1\leq i\leq k}$ is an orthonormal basis of the tangent space to $X$.

Since $Y$ is a Hadamard manifold,  
the continuous function  $\rho_h=d_{y_0}\circ h$ is $\mc C^\infty$
outside  $h^{-1}(y_0)$. 
Using Formula \eqref{eqndefh}, the harmonicity of $h$ and Lemma \ref{lemhessiandist}, we compute
the Laplacian on $X\setminus h^{-1}(y_0)$~:
\begin{equation*}
\label{eqnderohx}
\Delta\,\rho_h=\sum_{1\leq i\leq k} D^2d_{y_0}(D_{e_i}h,D_{e_i}h)\geq 0\, .
\end{equation*}
The function $\rho_h$ is the uniform limit when $\eps$ go to $0$ 
of the $\mc C^\infty$ functions 
$\rho_{h,\eps}:=(\eps^2+\rho_h^2)^{1/2}$.
We compute their Laplacian on $X\setminus h^{-1}(y_0)$~:
$$
\Delta \rho_{h,\eps}=
\frac{\rho_h}{(\eps^2+\rho_h^2)^{1/2}}\, \Delta \rho_h +
\frac{\eps^2}{(\eps^2+\rho_h^2)^{3/2}}\; \geq 0.
$$
It follows that the inequality $\Delta \rho_{h,\eps}\geq 0$ also holds on the whole $X$.

One finally gets $\Delta\rho_h\geq 0$ as a distribution on $X$
by letting $\eps$ go to $0$.
\end{proof}

Another crucial property of harmonic maps is the following bound 
for their differential due to Cheng.

\bl
\label{lemcheng}
Let $X$, $Y$ be two Hadamard manifolds with $-b^2\leq K_X\leq 0$. Let $k=\dim X$, 
$x_0$ be a point of $X$,  $r_0>0$
and let $h:B(x_0,r_0)\ra Y$ be a harmonic $\mc C^\infty$ map 
such that 
the  image 
$ h(B(x_0,r_0))$ lies in a ball of radius $R_0$. 
Then one has the bound
\begin{equation*}
\label{eqncheng}
\|Dh(x_0)\|\leq 2^5\,k\,\tfrac{1+br_0}{r_0}\, R_0\; .
\end{equation*}
\el
In the applications, we will use this inequality with $b=1$ and $r_0=1$.

\begin{proof}
This is a simplified version of \cite[Formula 2.9]{Cheng80} 
obtained by keeping track of the constants in the proof.
\end{proof}

\subsection{Smoothing quasiisometric maps}  
\label{secsmoothquasi}

\bq
The following proposition will allow us to assume 
in Theorem \ref{thdfxhx} that the quasiisometric map $f$ is $\mc C^\infty$
with bounded covariant derivatives.
\eq
\bp
\label{prosmoothquasi}
Let $X$, $Y$ be two  symmetric spaces of non positive curvature and 
$f:X\ra Y$ be a quasiisometric map.
Then there exists a $\mc C^\infty$ quasiisometric map $\widetilde{f}:X\ra Y$ 
within bounded distance from $f$ 
and  whose  covariant derivatives $D^p\widetilde{f}$
are bounded on $X$ for all  $p\geq 1$.
\ep

This regularized  map $\widetilde{f}$ will be constructed as follows.
Let $\al:\m R\ra \m R$ be a non-negative $\mc C^\infty$ function with support 
$[-1,1]$.
For $x$ in $X$, we introduce the positive finite measure on $X$
$$
\al_x:=(\al\circ d_x^2)\,\rmd {\rm vol}_X.
$$   
Let $\mu_x:= f_*\al_x$ denote the image measure on $Y$. 
It is defined, for any positive function $\ph$ on $Y$, by
$$
\mu_x(\ph)=\int_X\ph(f(z))\,\al(d^2(x,z))\rmd {\rm vol}_X(z)\, .
$$
We choose $\al$ so that each $\mu_x$ is a probability measure.
By homogeneity of $X$, this fact does not depend on the point $x$.
We will define $\widetilde{f}(x)\in Y$ to be the center of mass of the measure $\mu_x$.
To be more precise, we will need the following Lemma \ref{lemsmoothquasi}
which is an immediate consequence of Lemma \ref{lemhessiandist}.

\bl
\label{lemsmoothquasi}
For $x$ in $X$, let $Q_x$ be the function on $Y$
defined for $y$ in $Y$~by
\begin{equation}
\label{eqnqxy}
Q_x(y)\; =\; 
\int_Xd^2(y,f(z))\,\al(d^2(x,z))\,\rmd {\rm vol}_X(z)\, .
\end{equation}
For $x$ in $X$, the functions $Q_x$ are proper and uniformly strictly convex.
More precisely, for all 
$x$ in $X$ and $y$ in $Y$, the Hessian admits the lower bound 
$$
D^2_yQ_x\geq 2 \, g_Y
$$ 
where $g_Y$ is the Riemannian 
metric on $Y$.
\el

\begin{proof}[Proof of Proposition \ref{prosmoothquasi}]
For $x$ in $X$, we define the point $\widetilde{f}(x)\in Y$
to be the center of mass of $\mu_x$, i.e. to be the unique point where the function $Q_x$ 
reaches its infimum. Equivalently, the point $y=\widetilde{f}(x)\in Y$ is the unique critical point
of the function $Q_x$, i.e. it is defined by the implicit equation 
$$
D_yQ_x=0\; .
$$
Since the map $f$ is $c$-quasiisometric, the support of the measure $\mu_x$ lies in the ball
$B(f(x),2\, c)$. Since $Y$ is a Hadamard manifold, the balls of $Y$ are convex  (see \cite{Ballmann95}) so that
the center of mass $\widetilde{f}(x)$ also belongs to the ball $B(f(x),2\, c)$.
In particular, one has 
$$
d(f,\widetilde{f})\leq 2\, c.
$$ 

We now check  that the map $\widetilde{f}:X\ra Y$ is $\mc C^\infty$.
Since the Hessians $D^2_yQ_x$ are nondegenerate, 
this follows from the implicit
function theorem applied to the $\mc C^\infty$ map
$$
\Psi:(x,y)\in X\times Y
\longrightarrow 
D_yQ_x\in T^*Y\, .
$$

To prove that the first derivative of $\widetilde{f}$ is bounded on $X$,
we first notice that Lemma \ref{lemsmoothquasi}
ensures that the covariant derivative
$$
D_y\Psi(x,y)=D_y^2Q_x\in \mc L(T_yY,T_yY^*)
$$ 
is an invertible linear map with
$$
\|(D_y\Psi(x,y))^{-1}\|\leq 1/2\, .
$$
We also notice that, since the point $y=\widetilde{f}(x)$ 
is at distance at most $4c$ 
of all the points $f(z)$ with $z$ in the ball $B(x,1)$,
the norm 
$$
\|D_x\Psi(x,\widetilde{f}(x))\|
$$
is also uniformly bounded on $X$.
Hence the norm 
$\|D\widetilde{f}\|$ of the derivative of $\widetilde{f}$ is uniformly bounded on $X$.

For the same reason, since $Y$ is homogeneous, 
the norm of each covariant derivative
$$
\|D^p_xD^q_y\Psi(x,\widetilde{f}(x))\|
$$
is uniformly bounded on $X$ for $p,q\geq 0$.
Hence the norm of each covariant derivative
$\|D^p\widetilde{f}\|$ 
is uniformly bounded on $X$ for  $p\geq 1$.
\end{proof}

\subsection{Existence of harmonic maps}  
\label{secexistharm}

\bq
In this section we prove Theorem \ref{thdfxhx},
taking for granted Proposition \ref{proexistharm} below. 
\eq

Let $X$, $Y$ be rank one symmetric spaces and 
$f:X\ra Y$ be a $c$-quasi\-isometric $\mc C^\infty$ map whose first two covariant derivatives are bounded.

We fix a point $O$ in $X$. For $R>0$, we denote by $B_{_R}:=B(O,R)$ 
the closed ball in $X$ with center $O$ and radius $R$
and by $\partial B_{_R}$ the sphere that bounds $B_{_R}$. 
Since the manifold $Y$ is a Hadamard manifold, there exists
a unique harmonic map $h_{_R}:B_{_R}\ra Y$ satisfying the Dirichlet condition
$h_{_R}=f$ on the sphere
$\partial B_{_R}$. 
Moreover, this harmonic  map $h_{_R}$ is {\it energy minimizing}. 
This means that the  map $h_{_R}$ achieves the minimum of the energy functional
\begin{equation}
\label{eqnenergy}
E_{_R}(h):=\int_{B_{_R}}\| Dh(x)\|^2\,\rmd{\rm vol}_X(x)\, 
\end{equation}
among all $\mc C^1$ maps $g$ on the ball which agree with $f$ on the sphere $\partial B_{_R}$,
i.e. one has 
$$
E_{_R}(h_{_R})=\inf_gE_{_R}(g).
$$
These facts are due to Schoen 
(see \cite{Schoen77} or \cite[Thm 12.11]{EeLem78}).
Thanks to Schoen and Uhlenbeck in \cite{SchoenUhl82} and \cite{SchoenUhl83},
the harmonic map $h_{_R}$ is known to be $\mc C^\infty$ on the closed ball $B_{_R}$.
We denote by 
$$
d(h_{_R},f)=\sup_{x\in B(O,R)}d(h_{_R}(x),f(x))
$$
the distance between these two  maps.

The main point of this article  is to prove the following uniform estimate.

\bp
\label{proexistharm}
There exists a constant $M\geq 1$ such that, for any $R\geq 1$, one has
$d(h_{_R},f)\leq M$.
\ep

Even though the argument is very classical, 
we first explain how  to deduce our main theorem 
from this Proposition.

\begin{proof}[Proof of Theorem \ref{thdfxhx}]
As explained in Proposition \ref{prosmoothquasi}, 
we may also assume 
that the $c$-quasiisometric map $f$ is $\mc C^\infty$ with bounded covariant derivatives.  
Pick an increasing sequence of radii $R_n$ converging to $\infty$
and let $h_{R_n}:B_{R_n}\ra Y$ be the harmonic $\mc C^\infty$ map 
which agrees with $f$ on the sphere $\partial B_{R_n}$. 
Proposition \ref{proexistharm}
ensures that
the sequence of  maps $h_{R_n}$ is locally uniformly bounded.
More precisely there exists $M\geq 1$
such that, for all $S\geq 1$, for $n$ large enough, one has 
$$
h_{R_n}(B_{2S})\subset B(f(O),2\, cS+M).
$$ 
Using the Cheng Lemma \ref{lemcheng} with $b=1$ and $r_0=1$,
it follows that the derivatives are 
also uniformly bounded on the balls  $B_{_S}$.
More precisely one has, for all $S\geq 1$, for $n$ large enough, 
$$
\sup_{x\in B_{_S}}\|Dh_{R_n}(x)\|  \leq 2^6 k\,(2\,cS+M)\; .
$$
The Ascoli-Arzela theorem implies that, after extraction,
the sequence $h_{R_n}$ converges  uniformly on every ball $B_{_S}$
towards a continuous map $h:X\ra Y$.

By construction this limit  map $h$ 
stays within bounded distance from the  
quasiisometric map $f$. 
We claim that the limit map $h$ is harmonic. 
Indeed, the  harmonic maps $h_{R_n}$ are 
energy minimizing and, on each ball $B_{_S}$, the energies of $h_{R_n}$ are
uniformly bounded~:
$$
\limsup_{n\ra\infty}E_{S}(h_{R_n})<\infty\, .
$$
Hence the Luckhaus compactness theorem for energy minimizing harmonic maps
(see \cite[Section 2.9]{Simon96}) tells us that 
the limit map $h$ is also harmonic and energy minimizing.
\end{proof}

\begin{Rem}
{\rm By Li-Wang uniqueness theorem in \cite{LinWang08}, the harmonic map $h$
which stays within bounded distance from $f$ is unique.
Hence the above argument also proves that 
the whole family of harmonic maps $h_R$ converges to $h$
uniformly on the compact subsets of $X$
when $R$ goes to infinity.}  
\end{Rem}

\subsection{Boundary estimate}
\label{secboundary}

\bq
In this section we begin the proof of 
Proposition \ref{proexistharm}~: 
we  bound the distance between  $h_{_R}$ and $f$ near the
sphere  $\partial B_{_R}$.
\eq

\bp
\label{proboundary}
Let $X$, $Y$ be Hadamard manifolds and  $k=\dim X$. 
Assume that $-1\leq K_X\leq -a^2<0$. 
Let $c\geq 1$
and $f:X\ra Y$ be a $\mc C^\infty$ map with $\|Df(x)\|\leq c$ and 
$\|D^2f(x)\|\leq c$. 
Let  $O\in X$, $R>0$, $B_{_R}:=B(O,R)$. 

Let $h_{_R}:B_{_R}\ra Y$
be the harmonic $\mc C^\infty$ map  whose restriction to the sphere
$\partial B_{_R}$ is equal to $f$.
Then, for all $x$ in $B_{_R}$, one has 
\begin{equation}
\label{eqnboundary}
d(h_{_R}(x),f(x))\leq \tfrac{4kc^2}{a}\, d(x,\partial B_{_R})\; .
\end{equation}
\ep

An important feature of this upper bound is that 
it does not depend on the radius $R$,
provided the distance  $d(x,\partial B_{_R})$ remains bounded.
This is why we call \eqref{eqnboundary} the {\it boundary estimate}.
The proof relies on an idea of Jost in \cite[Section 4]{Jost84}

\begin{proof}
Let $x$ be a point in $B_{_R}$ and 
$w$ in $\partial B_{_R}$ 
such that $d(x,w)=d(x,\partial B_{_R})$. 
Since $h_{_R}(w)=f(w)$, the triangle inequality reads~as
\begin{equation}
\label{eqndfxhrx}
d(f(x),h_{_R}(x))\leq 
d(f(x),f(w)) + d(h_{_R}(w),h_{_R}(x))\; .
\end{equation}
The assumption on $f$ ensures that
\begin{equation}
\label{eqndfxfw}
d(f(x),f(w)) \leq c\, d(x,\partial B_{_R})\; .
\end{equation}
To estimate the other term, we choose a point $y_0$ on the geodesic ray starting from 
$h_{_R}(x)$ and passing by $h_{_R}(w)$. This choice of $y_0$ ensures that one has the equality
\begin{equation}
\label{eqndhrwhrx}
d(h_{_R}(w),h_{_R}(x))= d(h_{_R}(x),y_0)-d(h_{_R}(w),y_0)\; .
\end{equation}
We also choose $y_0$ far enough so that 
\begin{equation}
\label{eqnfz1}
F(z):= d(f(z),y_0)\geq 1
\;\; \mbox{\rm for all $z$ in $B_{_R}$.}
\end{equation}
This function $F$ is then $\mc C^\infty$
on the ball $B_{_R}$. 
Let $H:B_{_R}\ra \m R$ be the 
harmonic $\mc C^\infty$ function whose restriction to the sphere 
$\partial B_{_R}$ is equal to $F$.
By Lemma \ref{lemdefh}, since $h_{_R}$ is a harmonic map, the function 
$z\mapsto d(h_{_R}(z),y_0)$ is subharmonic on $B_{_R}$. 
Since this function is equal to $H$
on the sphere $\partial B_{_R}$, the maximum principle
ensures that
\begin{equation}
\label{eqndhrxy0}
d(h_{_R}(z),y_0)\leq H(z) 
\;\; \mbox{\rm for all $z$ in $B_{_R}$,}
\end{equation}
with equality for $z$ in $\partial B_{_R}$.
Combining \eqref{eqndhrwhrx} and \eqref{eqndhrxy0}, one gets 
\begin{equation}
\label{eqndhrxhrw}
d(h_{_R}(w),h_{_R}(x))\leq H(x)-H(w) \, .
\end{equation}
To estimate the right-hand side of \eqref{eqndhrxhrw}, 
we observe that the function 
$G:=F-H$ vanishes on  $\partial B_{_R}$ and has bounded Laplacian~: 
$$
|\Delta G|\leq 3k c^2.
$$
Indeed, using Formulas \eqref{eqndefh},
\eqref{eqnd2rhox0} and \eqref{eqnfz1}, one computes 
\begin{eqnarray*}
\label{eqndeg}
|\Delta G|=|\Delta (d_{y_0}\circ f)| \leq
k\|D^2d_{y_0}\|\|D f\|^2+ k\| Dd_{y_0}\| \|D^2f\|
\leq 3kc^2\, .
\end{eqnarray*}
Using Proposition \ref{proboundedlap}, one deduces that 
$$
|G(x)|\leq \tfrac{3kc^2}{a}\,d(x,\partial B_{_R})
$$  
and therefore, combining with 
\eqref{eqndfxhrx}, \eqref{eqndfxfw} and \eqref{eqndhrxhrw}, one concludes that
\begin{eqnarray*}
d(f(x),h_{_R}(x))&\leq& c\, d(x,\partial B_{_R})+
|G(x)|+|F(x)-F(w)|\\
&\leq&  
(\tfrac{3kc^2}{a}+2c)\,d(x,\partial B_{_R}).
\end{eqnarray*}
This proves \eqref{eqnboundary}.
\end{proof}


\section{Interior estimate}
\label{secinterior}

In this chapter we complete the proof of Proposition \ref{proexistharm}.  
We follow the strategy explained in the introduction (Section \ref{secstrategy}).

\subsection{Notation}  
\label{secnotation}

\bq
We first explain more precisely the notation and the assumptions 
that we will use in the
whole chapter.
\eq

Let $X$ and $Y$ be rank one symmetric spaces and $k=\dim X$. 
We start with a $\mc C^\infty$ 
quasiisometric map $f:X\ra Y$ all of whose covariant derivatives are bounded.
We fix a constant $c \geq 1$ such that, for all $x$, $x'$ in $X$,
one has 
\begin{equation}
\label{eqnquasiiso3}
\| Df(x)\|\leq c
\;\;{\rm , }\;\;\;
\| D^2f(x)\|\leq c
\;\;\;\;{\rm and }\;\;\;
\end{equation}
\begin{equation}
\label{eqnquasiiso2}
c^{-1}\, d(x,x')-c
\;\leq\; 
d(f(x),f(x'))
\;\leq\; 
c\, d(x,x')\; .
\end{equation}
Note that the additive constant $c$ on the right-hand side term of \eqref{eqnquasiiso}
has been removed since the derivative of $f$ is bounded by $c$.

We fix a point $O$ in $X$. For $R>0$, we 
introduce the harmonic $\mc C^\infty$ map $h_{_R}:B(O,R)\ra Y$ whose restriction to the sphere
$\partial B(O,R)$ is equal to $f$. 
We let 
\begin{equation*}
\label{eqnror}
\rho_{_R}:=\sup_{x\in B(O,R)}d(h_{_R}(x),f(x))\; .
\end{equation*}
We denote by $x_{_R}$ a point of $B(O,R)$ where the supremum is achieved~:
$$
d(h_{_R}(x_{_R}),f(x_{_R}))=\rho_{_R}\, .
$$
According to the boundary estimate in Proposition \ref{proboundary}, 
one has,
\begin{equation*}
\label{eqndx0dbr}
d(x_{_R},\partial B(O,R))\geq \tfrac{1}{8kc^2}\, \rho_{_R}\, .
\end{equation*}
When $\rho_{_R}$ is large enough, 
we introduce a  ball $B(x_{_R},r_{_R})$ with center 
$x_{_R}$, and whose radius 
$r_{_R}$ is a function of $R$ satisfying
\begin{equation}
\label{eqnrr}
1\leq r_{_R}\leq \tfrac{1}{16kc^2}\rho_{_R} \, .
\end{equation}
Note that this condition ensures the inclusion 
$B(x_{_R},r_{_R})\subset B(O,R\!-\!1).$
Later on, in section \ref{secboundi0}, we will assume that
$r_{_R}:= \rho_{_R}^{1/3}.$

We will focus on the restrictions of the  maps $f$ and $h_{_R}$ to this ball $B(x_{_R},r_{_R})$.
We will express the  maps $f$ and $h_{_R}$ through the  
polar exponential  coordinates $(\rho,v)$ in $Y$ centered at 
the point $y_{_R}:=f(x_{_R})$.
For $z$ in $B(x_{_R},r_{_R})$, we will thus write 
\begin{eqnarray*}
f(z)
&=&\exp_{y_{_R}}(\rho_f(z) v_f(z))\\
h_{_R}(z)
&=&
\exp_{y_{_R}}(\rho_h(z) v_h(z))\\
h_{_R}(x_{_R})
&=&
\exp_{y_{_R}}(\rho_{_R} v_{_R})
\end{eqnarray*}
where $\rho_f(z)\geq 0$, $\rho_h(z)\geq 0$ and where $v_f(z)$, $v_h(z)$ and $v_{_R}$ belong
to the unit sphere $T^1_{y_{_R}}Y$ of
the tangent space $T_{y_{_R}}Y$.
Note that $\rho_{h}$ and $v_{h}$ are shorthands for 
$\rho_{h_{_R}}$ and $v_{h_{_R}}$. For simplicity,
we do not write the dependance on $R$.

We denote by $[x_{_R},z]$ the geodesic segment between $x_ {_R}$ and $z$.

\begin{Def}
\label{defurvrwr}
We introduce the following subsets of the sphere  $S(x_{_R},r_{_R})$~:
\begin{eqnarray*}
\label{eqnur}
U_{_R}&=&
\{z\in S(x_{_R},r_{_R})\mid \rho_h(z)\geq \rho_{_R}-\tfrac{1}{2c}\, r_{_R}\}\\
\label{eqnvr}
V_{_R}&=&
\{z\in S(x_{_R},r_{_R})\mid \rho_h(z_t)\geq \rho_{_R}/2
\;\;\; \mbox{\rm for all\; $z_t$ in $[x_{_R},z]$}\, \}\\
\label{eqnwr}
W_{_R}&=& U_{_R}\cap V_{_R}\, .
\end{eqnarray*}
\end{Def}

\subsection{Measure estimate}  
\label{secmeasure}

\bq
We first notice that  
one can control the size of $\rho_h(z)$ and of $Dh_{_R}(z)$ on the ball $B(x_{_R},r_{_R})$.
We will then give a lower bound for the measure of $W_{_R}$.
\eq

\bl
\label{lemrhx}
Assume \eqref{eqnrr}. For $z$ in $B(x_{_R},r_{_R})$, one has 
\begin{equation*}
\label{eqnrohx}
\rho_h(z)\leq \rho_{_R}+c\,r_{_R}.
\end{equation*}
\el

\begin{proof}
The triangle inequality and \eqref{eqnquasiiso2} give,
for $z$ in $B(x_{_R},r_{_R})$,
\vspace{1ex}

\mbox{}\hfill$
\rho_h(z)\leq d(h_{_R}(z),f(z))+d(f(z),y_{_R})\leq \rho_{_R}+c\,r_{_R}\, .
$\hfill
\end{proof}

\bl
\label{lemdhx}
Assume \eqref{eqnrr}. For $z$ in $B(x_{_R},r_{_R})$, one has 
\begin{equation*}
\label{eqndhx}
\|Dh_{_R}(z)\|\leq 2^8\, k\,\rho_{_R}.
\end{equation*}
\el

\begin{proof}
For all $z$, $z'$ in $B(O,R)$ with $d(z,z')\leq 1$, the triangle inequality and \eqref{eqnquasiiso2}
yield
\begin{eqnarray*}
d(h_{_R}(z),h_{_R}(z'))&\leq &
d(h_{_R}(z),f(z))+d(f(z),f(z'))+d(f(z'),h_{_R}(z'))\\
&\leq &\rho_{_R}+ c +\rho_{_R}\leq 4\,\rho_{_R}\, . 
\end{eqnarray*}
Applying Cheng's lemma \ref{lemcheng} with $b=1$ and $r_0=1$,
one gets for all $z$ in $B(O,R\! -\! 1)$ the bound
$\|Dh_{_R}(z)\|\leq 2^8\, k\,\rho_{_R}$.
\end{proof}

We now give a lower bound for the measure of $W_{_R}$.
We will denote by the same letter 
$\si$ the probability measure on each sphere $ S(x_{_R},r_{_R})$
that is invariant under all the isometries of $X$ that fix the point $x_{_R}$.

\begin{Lem}
\label{lemsiwr}
Assume \eqref{eqnrr}. Then one has
\begin{eqnarray}
\label{eqnsiwr}
\si(W_{_R})&\geq &\frac{1}{3\,c^2}-2^{12}k\,c\,\frac{r_{_R}^2}{\rho_{_R}}\; .
\end{eqnarray}
\end{Lem}

\begin{proof}
The proof relies on the subharmonicity of the function $\rho_h$ 
on the ball $B(x_{_R},r_{_R})$ (see Lemma \ref{lemdefh}). 
We claim that
\begin{equation}
\label{eqnintrhru}
\int_{S(x_{_R},r_{_R})}(\rho_h(z)-\rho_{_R})\rmd\si(z) \geq 0\; .
\end{equation} 
Let us give a short proof of this special case of the Green formula. 
Since $X$ is a symmetric space, the group $\Ga$ of isometries of $X$
that fix  the point $x_{_R}$ is a compact group which acts transitively on the spheres
$S(x_{_R},t)$. Let $\rmd \ga$ be the Haar probability measure on $\Ga$. 
The function $F:=\int_{\Ga}\rho_h\circ \ga \,\rmd \ga$, defined as the average of the translates of 
the function $\rho_h$ under $\Ga$,  is equal to a constant $F_t$ 
on each sphere $S(x_{_R},t)$ of radius $t\leq r_{_R}$.
By the maximum principle applied to this subharmonic function $F$, one gets $F_0\leq F_t$ for all $t\leq r_{_R}$. Since $F_0=\rho_h(x_{_R})=\rho_{_R}$,
this proves \eqref{eqnintrhru}.
\vs

\noi{\it First step.} We prove
\begin{eqnarray}
\label{eqnsiur}
\si(U_{_R})&\geq &\frac{1}{3\,c^2}\; .
\end{eqnarray}
By Lemma \ref{lemrhx}, the function $\rho_h$ is bounded by $\rho_{_R}+c\,r_{_R}$, 
hence Equation \eqref{eqnintrhru} implies
\begin{equation*}
c\, r_{_R}\,\si(U_{_R})-\frac{r_{_R}}{2c}\,(1-\si(U_{_R}))\geq 0\; 
\end{equation*} 
so that $\si(U_{_R})\geq (1+2\,c^2)^{-1}\geq c^{-2}/3$.
\vs

\noi{\it Second step.} We  prove 
\begin{eqnarray}
\label{eqnsivr}
\si(V_{_R})&\geq &1-2^{12}k\,c\,\frac{r_{_R}^2}{\rho_{_R}}\; .
\end{eqnarray}
For $z$ in the complementary subset $V_{_R}^c\subset S(x_{_R},r_{_R})$, we define
\begin{eqnarray*}
t_{z}&:=&\inf\{ t\in[0, r_{_R}]\;\;\mid\; \rho_h(z_t)=\tfrac12\,\rho_{_R}\}\, ,\\
s_{z}&:=&\sup\{ t\in [0, t_z] \mid\; \rho_h(z_t)=\tfrac34\,\rho_{_R}\}\, .
\end{eqnarray*} 
We claim that, for each $z$ in $V_{_R}^c$, one has 
\begin{equation}
\label{eqntusu}
t_z-s_z\geq 2^{-10} k^{-1}\, .
\end{equation}
Indeed, the length of the curve $t\mapsto h_{_R}(z_t)$ between 
$t=s_z$ and $t=t_z$ is at least 
$\frac{\rho_{_R}}{4}$. Hence, using  Lemma \ref{lemdhx}, one gets
\begin{eqnarray*}
\frac{\rho_{_R}}{4} 
&\leq & (t_z-s_z)\,\sup_{B(x_{_R},r_{_R})}\| Dh_{_R}\|
\leq 2^8\,k\,(t_z-s_z)\rho_{_R}\, , 
\end{eqnarray*}  
which prove \eqref{eqntusu}. 

The Green formula also  gives the following variation of
\eqref{eqnintrhru} 
\begin{equation}
\label{eqnintrhtu}
\int_{S(x_{_R},r_{_R})}\int_0^{r_{_R}}(\rho_h(z_t)-\rho_{_R})\rmd t\,\rmd\si(z) \geq 0\; .
\end{equation} 
By Lemma \ref{lemrhx}, the function $\rho_h$ is bounded by $\rho_{_R}+c\,r_{_R}$, 
hence Equation \eqref{eqnintrhtu} implies
\begin{equation}
\label{eqnintvr}
c\, r_{_R}^2+\int_{V_{_R}^c}\int_{s_z}^{t_z}(\rho_h(z_t)-\rho_{_R})\rmd t\,\rmd\si(u) \geq 0\; .
\end{equation} 
Using the bound $\rho_h(z_t)\leq \tfrac34\,\rho_{_R}$, 
for all $t$ in the interval $[s_z,t_z]$, one deduces from  \eqref{eqntusu} and \eqref{eqnintvr}
that
$$
c\,r_{_R}^2-2^{-10} k^{-1}\tfrac{\rho_{_R}}{4}\,\si(V_{_R}^c)\geq 0\; .
$$
This proves \eqref{eqnsivr}.

Since $W_{_R}=U_{_R}\cap V_{_R}$, the bound \eqref{eqnsiwr} follows from \eqref{eqnsiur} and \eqref{eqnsivr}.
\end{proof}

\subsection{Upper bound for $\theta(v_f(z),v_h(z))$}
\label{secboundi1}
\bq 
For all $v$ in $U_R$, we give an upper bound for the angle  between $v_f(z)$ and $v_h(z)$.
\eq

For two vectors $v_1$, $v_2$ of the unit sphere $T^1_{y_{_R}}Y$ of the tangent space $T_{y_{_R}}Y$, we denote by 
$\theta(v_1, v_2)$  the angle between these two vectors. 

\bl
\label{lemI1}
Assume \eqref{eqnrr}. Then, for  $z$ in $U_{_R}$, one has
\begin{eqnarray}
\label{eqnI1b}
\theta(v_f(z),v_h(z))&\leq & 4\,e^{\frac{c}{4}}\,e^{-\frac{r_{_R}}{8\, c}}.
\end{eqnarray} 
\el

\begin{proof}
For $z$ in $U_{_R}$, we consider the triangle with vertices 
$y_{_R}$, $f(z)$ and $h_{_R}(z)$. Its side lengths satisfy
\begin{eqnarray*}
\label{eqnsidel}
\ell_0&:=&d(h_{_R}(z),f(z))\leq \rho_{_R}
\hspace{5ex} \mbox{\rm by definition of $\rho_{_R}$,}\\
\ell_1&:=&\rho_f(z)\geq \tfrac{1}{c}\,r_{_R}- c
\hspace{8.5ex} \mbox{\rm since $f$ is $c$-quasiisometric,}\\
\ell_2&:=&\rho_h(z)\geq \rho_{_R}-\tfrac{1}{2c}\, r_{_R}
\hspace{6ex} \mbox{\rm by definition of $U_R$.}
\end{eqnarray*}
Since $K_Y\leq -1/4$, applying Lemma \ref{lemcomparison} with $a=\frac12$, one gets
\vspace{1ex}

\mbox{}\hfill
$
\theta(v_f(z)),v_h(z))\leq 
4\, e^{-\frac{1}{4}(\ell_1+\ell_2-\ell_0)}\leq 
4\, e^{\frac{c}{4}}\,e^{-\frac{r_{_R}}{8\, c}}\, .
$ 
\hfill
\end{proof}

\subsection{Upper bound for $\theta(v_h(z),v_{_R})$}
\label{secboundi2}
\bq 
For all $v$ in $V_R$, we give an upper bound for the angle between $v_h(z)$ and $v_{_R}$.
\eq  

\bl
\label{lemI2}
Assume \eqref{eqnrr}. Then, for  $z$ in $V_{_R}$, one has
\begin{eqnarray}
\label{eqnI2b}
\theta(v_h(z),v_{_R})&\leq &
\frac{8\,\rho_{_R}^2}{{\rm sinh}(\rho_{_R}/4)}.
\end{eqnarray} 
\el

\begin{proof}
Let us first sketch the proof.
We recall that the curve $t\mapsto z_t$, for $0\leq t\leq r_{_R}$,  
is the geodesic segment
between $x_{_R}$ and  $z$.
By definition, for each  $z$ in $V_{_R}$, the curve $t\mapsto h_{_R}(z_t)$ lies outside of the ball $B(y_{_R},\rho_{_R}/2)$
and by Cheng's bound on $\| Dh_{_R}(z_t)\|$ 
one controls the length of this curve.

We now  detail the argument. For $z$ in $V_{_R}$, we have the inequality
\begin{eqnarray*}
\label{eqnIR2a}
\theta(v_h(z), v_{_R})
&\leq &
r_{_R}\,\sup_{0\leq t\,\leq \,r_{_R}}\| Dv_h(z_t)\|\, .
\end{eqnarray*}
Since $K_Y\leq -1/4$ the Alexandrov triangle comparison theorem and the Gauss lemma (\cite[2.93]{GHL04}) yield,
for $y$ in $Y\smallsetminus\{ y_{_R}\}$,
\begin{equation*}
\label{eqndvy}
2\,{\rm sinh}(\rho(y)/2)\,\| Dv(y)\|\leq 1\, ,
\end{equation*}
where $(\rho(y), v(y))\in\; ]0,\infty[\times T^1_{y_{_R}}Y$ are the polar exponential coordinates on $Y$ centered at $y_{_R}$.
Since $\rho_h=\rho\circ h$ and $v_h=v\circ h$, one gets,  for  $x$ in $B(O,R)$ with $h_{_R}(x)\neq y_{_R}$, 
\begin{equation*}
\label{eqndvhdh}
2\,{\rm sinh}(\rho_h(x)/2)\,\| Dv_h(x)\|\leq \| Dh_R (x)\|\, .
\end{equation*}
Hence 
since the point $z$ belongs to $V_{_R}$
one deduces 
\begin{eqnarray*}
\label{eqnIR2b}
\theta(v_h(z), v_{_R})
&\leq &
\frac{r_{_R}}{2\,{\rm sinh}(\rho_{_R}/4)}\,\sup_{0\leq t\,\leq \,r_{_R}}\| Dh_R(z_t)\|\, .
\end{eqnarray*}
Hence using  Lemma \ref{lemdhx} one gets
\begin{eqnarray*}
\label{eqnIR2c}
\theta(v_h(z), v_{_R})
&\leq &
2^8\, k\,\rho_{_R}\,\frac{r_{_R}}{2\,{\rm sinh}(\rho_{_R}/4)}\,.
\end{eqnarray*}
Using \eqref{eqnrr} this finally gives \eqref{eqnI2b}.
\end{proof}

\subsection{Lower bound for $\theta(v_f(z),v_{_R})$}
\label{secboundi0}
\bq 
When $\rho_{_R}$ is large enough, we find a point $z=z_{_R}$ in $W_{_R}$ for which the angle between $v_f(z)$ and $v_{_R}$ is bounded below.
\eq  

For a subset $W$ of the unit sphere 
$ T^1_{y_{_R}}Y$, 
we denote by 
$$
{\rm diam}(W):=\sup(\{\theta(v,v')\mid v,v'\in W\})
$$ 
the diameter of $W$.

\bl
\label{lemI0} 
Assume that there exists a sequence of radii $R$ going to infinity 
such that $\rho_{_R}$ goes to infinity.
Then, choosing $r_{_R}:=\rho_{_R}^{1/3}$, 
the diameters  
$$
{\rm diam}(\{v_f(z)\mid z\in W_{_R}\})
$$ 
do not converge to $0$
along this sequence. 
\el

\begin{proof}
Let $\si_0:=\frac{1}{4c^2}$. According to Lemma \ref{lemsiwr}, one has
\begin{equation*}
\label{eqnsiwro}
\liminf_{R\ra\infty}
\si( W_{_R})>\si_0 >0\; .
\end{equation*}

There exists  $\eps_0>0$ such that every subset $W$ of the Euclidean sphere $\m S^{k-1}$ whose 
normalized measure 
is at least  $\si_0$ contains two  points 
whose angle  is at least $\eps_0$.

Hence
if $R$ and $\rho_{_R}$ are large enough
one can find 
$z_1$, $z_2$ in $W_{_R}$ such that 
\begin{equation}
\label{eqnu1u2}
\theta_{x_{_R}}(z_1,z_2)\geq\eps_0
\;\;\;\; {\rm and} \;\;\;\;\;
r_{_R}\geq\, \frac{(A+1)c}{ \sin^2(\eps_0/2)}\, ,
\end{equation}
where $\theta_{x_{_R}}(z_1,z_2)$ is the angle 
between $z_1$ and $z_2$ seen from $x_{_R}$,
and where $A$ is the  constant given 
by Lemma \ref{lemproduct}.
According to  Lemma \ref{lemcomparison}.$a$, we infer that
\begin{equation*}
\label{eqnxrzz}
\min( (x_{_R}|z_1)_{z_2},(x_{_R}|z_2)_{z_1}))\geq (A+1)c.
\end{equation*}
Using then Lemma \ref{lemproduct}, one gets 
\begin{equation}
\label{eqnyrfxfx}
\min( (y_{_R}|f(z_1))_{f(z_2)},(y_{_R}|f(z_2))_{f(z_1)}))\geq 1.
\end{equation}
We now have the inequalities
\begin{eqnarray*}
\theta(v_f(z_1),v_f(z_2))
&\geq & 
e^{-(f(z_1)|f(z_2))_{y_{_R}}}
\hspace{6.5ex} \mbox{\rm by  Lemma \ref{lemcomparison}.$c$ and \eqref{eqnyrfxfx}}\\
&\geq &
e^{-A}\, e^{-c\,(z_1|z_2)_{x_{_R}}}
\hspace{6ex} \mbox{\rm by Lemma \ref{lemproduct}}\\
&\geq &
e^{-A}\, (\eps_0/4)^{2c}
\hspace{9.5ex} \mbox{\rm by Lemma \ref{lemcomparison}.$b$
and \eqref{eqnu1u2}.}\\
\end{eqnarray*}
This proves our claim.
\end{proof}

\begin{proof}[End of the proof of Proposition \ref{proexistharm}] 
Assume that there exists a sequence of radii $R$ going to infinity such that
$\rho_{_R}$ goes also to infinity. We set $r_{_R}=\rho_{_R}^{1/3}$. Using  Lemmas \ref{lemI1} and \ref{lemI2} and the triangle inequality, one gets 
\begin{equation}
\label{eqnsupdvfvr}
\lim_{R\ra\infty}
\sup_{z\in W_{_R}} 
\theta(v_f(z),v_{_R})=0\, .
\end{equation}
This contradicts Lemma \ref{lemI0}. 
\end{proof}


\small{
\bibliographystyle{plain}
\bibliography{harmonic}
}
\vspace{2em}

{\small\noindent Y. Benoist  \& D. Hulin,\;
CNRS \& Universit\'e Paris-Sud, 
Orsay 91405 France\\
\noindent yves.benoist@math.u-psud.fr \;\&\; dominique.hulin@math.u-psud.fr}

\end{document}